\documentclass[12pt,a4paper]{amsart}
\usepackage[top=32mm, bottom=32mm, left=32mm, right=32mm]{geometry}
\usepackage{mathptmx}
\usepackage{graphicx}
\usepackage{amssymb}
\usepackage{amsmath}
\usepackage{amsthm}
\usepackage{amscd}
\usepackage{color}
\usepackage[colorlinks=true,citecolor=blue]{hyperref}
\newtheorem{thm}{Theorem}[section]
\newtheorem{cor}[thm]{Corollary}
\newtheorem{lem}[thm]{Lemma}
\newtheorem{pro}[thm]{Proposition}
\newtheorem{rem}[thm]{Remark}
\theoremstyle{definition}
\newtheorem{de}[thm]{Definition}
\numberwithin{equation}{section}
\newcommand{\ep}{\varepsilon}
\DeclareMathOperator{\supp}{supp}

\begin{document}
\title[Stable sets and chaos]{Stable sets and mean Li-Yorke chaos in positive entropy systems}
\author{Wen Huang, Jian Li and Xiangdong Ye}
\address[W.~Huang, X.~Ye]{Wu Wen-Tsun Key Laboratory of Mathematics, USTC, Chinese Academy of Sciences,
Department of Mathematics, University of Science and Technology of China,
Hefei, Anhui, 230026, P.R. China} \email{wenh@mail.ustc.edu.cn,
yexd@ustc.edu.cn}
\address[J.~Li]{Department of Mathematics, Shantou University, Shantou, Guangdong 515063, P.R. China}
\email{lijian09@mail.ustc.edu.cn}
\thanks{The authors were supported by NNSF of China (11371339). The first author was partially supported by
Fok Ying Tung Education Foundation and the Fundamental Research
Funds for the Central Universities (WK0010000014) and NNSF for
Distinguished Young Scholar (11225105). The second author was
partially supported by STU Scientific Research Foundation for
Talents (NTF12021).}



\begin{abstract}
It is shown that in a topological dynamical system with positive
entropy, there is a measure-theoretically ``rather big'' set such
that a multivariant version of mean Li-Yorke chaos happens on the
closure of the stable or unstable set of any point from the set. It
is also proved that the intersections of the sets of asymptotic
tuples and mean Li-Yorke tuples with the set of topological entropy
tuples are dense in the set of topological entropy tuples
respectively.
\end{abstract}

\maketitle

\section{Introduction}
Throughout this paper, by a {\it topological dynamical system}
$(X,T)$ (t.d.s. for short) we mean  a compact metric space $X$ with
a homeomorphism $T$ from $X$ onto itself. The metric on $X$ is
denoted by $d$. For a t.d.s. $(X,T)$, the stable set of a point
$x\in X$ is defined as
\[ W^s(x,T)=\big\{y\in X: \lim_{n\to +\infty} d(T^nx, T^ny)=0\big\} \]
and the unstable set of $x$ is defined as
\[W^u(x,T)=\big\{ y\in X:\lim_{n\to +\infty} d(T^{-n}x, T^{-n}y)=0\big\}.\]
Clearly, $W^s(x,T)=W^u(x,T^{-1})$ and $W^u(x,T)=W^s(x,T^{-1})$ for
each $x\in X$. Stable and unstable sets play a big role in the study
of smooth dynamical systems. Recently, there are many results
related to the chaotic behavior and stable (unstable) sets in
positive entropy systems (cf. \cite{BGKM,BHS,BH,D, DY, FHYZ,H}). So
it indicates that stable and unstable sets are also important in the
study of topological dynamical systems.

The chaotic behavior of a t.d.s. reflects the complexity of a t.d.s.
Among the various definitions of chaos, Devaney's chaos, Li-Yorke
chaos, positive entropy and the distributional chaos are the most
popular ones. The implication among them has attracted a lot of
attention. It is shown by Huang and Ye \cite{HY02} that Devaney's
chaos implies Li-Yorke one by proving that a non-periodic transitive
t.d.s. with a periodic point is chaotic in the sense of Li and
Yorke. In~\cite{BGKM}, Blanchard, Glasner, Kolyada and Maass proved
that positive entropy also implies Li-Yorke chaos, that is if a
t.d.s. $(X,T)$ has positive entropy then there exists an uncountable
subset $S$ of $X$ such that for any two distinct points $x,y\in S$,
\[\liminf_{n\to +\infty}\ d(T^{n}x,T^{n}y)=0 \text{ and }
 \limsup_{n\to +\infty}\ d(T^{n}x,T^{n}y) >0.\] We remark that the
authors obtained this result using ergodic method, and for a
combinatorial proof see \cite{KL}. Moreover, the result also holds
for sofic group actions by Kerr and Li \cite[Corollary 8.4]{KL12}.

In~\cite{BHS} Blanchard, Host and Ruette investigated the question
if positive entropy implies the existence of non-diagonal asymptotic
pairs. Among other things the authors showed that for positive
entropy systems, many stable sets are not stable under $T^{-1}$.
More precisely, if a $T$-invariant ergodic measure $\mu$ has
positive entropy, then there exists $\eta>0$ such that for
$\mu$-a.e. $x\in X$, one can find an uncountable subset
 $F_x$ of $W^s(x,T)$ satisfying that for any $y\in F_x$,
 \[\liminf_{n\to +\infty}\ d(T^{-n}x,T^{-n}y)=0 \text{ and }
 \limsup_{n\to +\infty}\ d(T^{-n}x,T^{-n}y) \geq \eta.\]

Distributional chaos was introduced in \cite{SS}, and there are at
least three versions of distributional chaos in the literature (DC1,
DC2 and DC3), see \cite{BJM} for the details. It is known that
positive entropy does not imply DC1 chaos \cite{Pi}. In \cite{S}
Sm\'{\i}tal conjectured that positive entropy implies DC2 chaos, and
Oprocha showed that this conjecture holds for minimal uniformly
positive entropy systems~\cite{O}. Very recently, Downarowicz
\cite{D} proved that positive entropy indeed implies DC2 chaos.
Precisely, if a t.d.s. $(X,T)$ has positive entropy then there
exists an uncountable subset $S$ of $X$ such that for any two
distinct points $x,y\in S$,
\[\liminf_{N\to +\infty}\ \frac{1}{N}\sum_{i=1}^{N}d(T^{i}x,T^{i}y)=0 \text{ and }
 \limsup_{N\to +\infty}\  \frac{1}{N}\sum_{i=1}^{N}d(T^{i}x,T^{i}y) >0.\]

Entropy pairs were introduced and studied in \cite{Bl} by Blanchard,
and the notion was extended to entropy tuples in \cite{HY} by Huang
and Ye. Nowadays, the idea of using various tuples to obtain
dynamical properties is widely accepted and developed, see for
example the surveys \cite{GY, OZ}. In~\cite{Xiong2005}, Xiong
extended Huang-Ye's result related to Li-Yorke chaos to the
multivariant version of Li-Yorke chaos. That is, if $(X,T)$ is a
non-periodic transitive system with a fixed point, then there exists
a dense Mycielski subset (countable union of Cantor sets) $S$ of $X$
such that for any $n\geq 2$ and any $n$ distinct points
$x_1,\dotsc,x_n\in S$, one has
\[\liminf_{k\to +\infty}\max_{1\leq i<j\leq n}d\big(T^k x_i , T^k x_j)=0 \
\text{and}\ \limsup_{k\to +\infty}\min_{1\leq i<j\leq n}d\big(T^k
x_i , T^k x_j)>0.\]

Motivated by the above ideas, results and also by the consideration
of Ornstein and Weiss in \cite{OW} (where they defined the notion of
mean distality) we introduce the following multivariant version of
mean Li-Yorke chaos.


A tuple $(x_1,\dotsc,x_n)\in X^{(n)}$ is called \emph{mean Li-Yorke $n$-scrambled (with modulus $\eta$)} if
\[\liminf_{N\to +\infty}\ \frac{1}{N}\sum_{k=1}^N\max_{1\leq i<j\leq n}d\big(T^k x_i , T^k x_j)=0\qquad\ \]
and
\[\limsup_{N\to +\infty}\ \frac{1}{N}\sum_{k=1}^N\min_{1\leq i<j\leq n}d\big(T^k x_i , T^k x_j)\geq \eta>0.\]

A subset  $S$ of $X$ is called \emph{mean Li-Yorke $n$-scrambled (with modulus $\eta$)} if
any $n$ distinct points in $S$ form a mean Li-Yorke $n$-scrambled tuple (with modulus $\eta$).
The system $(X,T)$ is called \emph{mean Li-Yorke $n$-chaotic} if there is an uncountable
mean Li-Yorke $n$-scrambled set.

The aim of this paper is to investigate the mean Li-Yorke chaos
appearing in the closure of stable or unstable sets of a dynamical
system with positive entropy, and the relationship among the sets of
asymptotic tuples, mean Li-Yorke tuples and topological entropy
tuples. More specifically, we have the following main results.

\begin{thm} \label{thm-1}
Let $(X,T)$ be a t.d.s. with an ergodic invariant measure $\mu$ of
positive entropy. Then there exists a sequence of positive numbers $\{\eta_n\}_{n=2}^{+\infty}$
satisfying that for $\mu$-a.e. $x\in X$, there exists a Mycielski
subset $K_x\subseteq \overline{W^s(x,T)}\cap \overline{W^u(x,T)}$
such that for every $n\geq 2$, $K_x$ is a mean Li-Yorke
$n$-scrambled set with modulus $\eta_n$ both for $T$ and for
$T^{-1}$.
\end{thm}

It is observed in \cite{D} that a pair is DC2-scrambled if and only
if it is mean Li-Yorke $2$-scrambled in our sense. We remark that
the key tool in the proof of Theorem \ref{thm-1} is the excellent
partition constructed in \cite[Lemma~4]{BHS} which is different from
the one used in \cite{D}.  So among other things for $n=2$ we also
obtain a new proof of Downarowicz's result in \cite{D}.

\begin{thm} \label{thm-2} Let $(X,T)$ be a t.d.s. with positive entropy. Then
for any $n\ge 2$
\begin{enumerate}

\item the intersection of the set of asymptotic $n$-tuples with the set of
topological entropy $n$-tuples is dense in the set of topological
entropy $n$-tuples;

\item the intersection of the set of mean Li-Yorke $n$-tuples with the set of
topological entropy $n$-tuples is dense in the set of topological
entropy $n$-tuples.

\end{enumerate}

\end{thm}
Note that in \cite{BHS} and \cite{Gl} the authors showed that the
closure of the set of asymptotic pairs contains the set of entropy
pairs, and the intersection of the set of proximal pairs with the
set of entropy pairs is dense in the set of entropy pairs
respectively. Moreover, the authors in \cite{KL} proved that the
intersection of the set of Li-Yorke pairs with the set of entropy
pairs is dense in the set of entropy pairs. Thus Theorem \ref{thm-2}
extends the mentioned results in \cite{BHS}, \cite{Gl} and
\cite{KL}.

\medskip
Note also that our main results (Theorems~\ref{thm-1} and~\ref{thm-2}) also hold for continuous maps with
an obvious modification, see Theorems~\ref{conti} and~\ref{cor-310}.



\medskip \noindent {\bf Acknowledgement}:
We would like to thank Hanfeng Li for useful comments.
We also thank the anonymous referee for his/her helpful suggestions concerning this paper.

\section{Preliminaries}
For a t.d.s. $(X,T)$, denote by $\mathcal{B}_X$ the $\sigma$-algebra
of Borel subsets of $X$. A \emph{cover} of $X$ is a family of Borel
subsets of $X$ whose union is $X$. An \emph{open cover} is the one
consisting of open sets. A \emph{partition} of $X$ is a cover of $X$
by pairwise disjoint sets. Given a partition $\alpha$ of $X$ and
$x\in X$, denote by $\alpha(x)$ the atom of $\alpha$ containing $x$.

We denote the collection of finite partitions, finite covers and
finite open covers of $X$  by $\mathcal{P}_X$, $\mathcal{C}_X$ and
$\mathcal{C}^o_X$, respectively. Let $\mathcal{U}$ and $\mathcal{V}$
be two covers. Their \emph{join} $\mathcal{U}\bigvee \mathcal{V}$ is
the cover $\{U\cap V:\ U\in\mathcal{U},V\in\mathcal{V}\}$. For
$\mathcal{U}\in\mathcal{C}_X$, we define $N(\mathcal{U})$ as the
minimum among the cardinalities of the subcovers of $\mathcal{U}$.
The \emph{topological entropy} of $\mathcal{U}$ with respect to $T$
is
\begin{equation*}
h_{\text{top}}(T,\mathcal{U})=\lim_{N\to +\infty}{\frac1N}\log N\bigg(\bigvee_{i=0}^{N-1}T^{-i}\mathcal{U}\bigg).
\end{equation*}
The \emph{topological entropy} of $(X,T)$ is defined by
$h_{\text{top}}(T)=\sup_{\mathcal{U}\in\mathcal{C}^o_X}h_{\text{top}}(T,\mathcal{U})$.

Let $\mathcal{M}(X)$, $\mathcal{M}(X, T)$ and
$\mathcal{M}^{e}(X,T)$ be the collections of all Borel probability
measures, $T$-invariant Borel probability measures  and
$T$-invariant ergodic measures on $X$, respectively. Then
$\mathcal{M}(X)$ and $\mathcal{M}(X,T)$ are convex, compact metric
spaces when endowed with the weak$^*$-topology.

For any given $\alpha \in \mathcal{P}_X$ and  $\mu\in {\mathcal
M}(X)$, let $H_{\mu}(\alpha)=\sum_{A\in \alpha} -\mu(A) \log
\mu(A)$.  When $\mu\in \mathcal{M}(X,T)$, we define the
\emph{topological entropy} of $\alpha$ with respect to $\mu$ as
\[
h_\mu(T,\alpha)=\lim_{n\rightarrow\infty}\frac{1}{n} H_\mu\bigg(\bigvee_{i=0}^{n-1}T^{-i}\alpha\bigg).
\]
The \emph{measure-theoretic entropy} of $\mu$ is defined by
$h_\mu(T)=\sup_{\alpha \in \mathcal{P}_X} h_\mu(T,\alpha)$.

The relation between topological entropy and measure-theoretic
entropy is the following well known variational principle:
$h_{\text{top}}(T)=\sup_{\mu\in \mathcal{M}^e(X, T)} h_\mu(T)$.

Let $(X,T)$ be a t.d.s., $\mu\in \mathcal{M}(X,T)$ and
$\mathcal{B}_\mu$ be the completion of $\mathcal{B}_X$ under the
measure $\mu$. Then $(X,\mathcal{B}_\mu,\mu,T)$ is a Lebesgue
system. If $\{ \alpha_i\}_{i\in I}$ is a countable family of finite
partitions of $X$, the partition $\alpha=\bigvee_{i\in I}\alpha_i$
is called a {\it measurable partition}. The sets $A\in
\mathcal{B}_\mu$, which are unions of atoms of $\alpha$, form a
sub-$\sigma$-algebra $\mathcal{B}_\mu$ denoted by $\widehat{\alpha}$
or $\alpha$ if there is no ambiguity. Every sub-$\sigma$-algebra of
$\mathcal{B}_\mu$ coincides with a $\sigma$-algebra constructed in
this way (mod $\mu$).

For a measurable partition $\alpha$, put
$\alpha^-=\bigvee_{n=1}^{+\infty} T^{-n}\alpha$ and
$\alpha^T=\bigvee_{n=-\infty}^{+\infty} T^{-n}\alpha$. Define in
the same way $\mathcal{F}^-$ and $\mathcal{F}^T$ if $\mathcal{F}$
is a sub-$\sigma$-algebra of $\mathcal{B}_\mu$. It is clear that
for a measurable partition $\alpha$ of $X$,
$\widehat{\alpha^-}=(\widehat{\alpha})^-$ and
$\widehat{\alpha^T}=(\widehat{\alpha})^T$(mod $\mu$).

Let $\mathcal{F}$ be a  sub-$\sigma$-algebra of $\mathcal{B}_\mu$ and
$\alpha$ be a measurable partition of $X$ with $\widehat{\alpha}=\mathcal{F}$ (mod $\mu$).
Then $\mu$ can be disintegrated over $\mathcal{F}$ as
\[\mu=\int_X \mu_x d \mu(x),\]
where $\mu_x\in \mathcal{M}(X)$  and $\mu_x(\alpha(x))=1$ for $\mu$-a.e. $x\in X$.
The disintegration can be characterized by the
properties \eqref{meas1} and \eqref{meas3} as follows:
\begin{align}
&\text{for every } f \in L^1(X,\mathcal{B}_X,\mu),\
f \in L^1(X,\mathcal{B}_X,\mu_x)\ \text{for $\mu$-a.e. } x\in X, \label{meas1}\\
& \text{and the map $x \mapsto \int_X  f(y)\,d\mu_x(y)$ is in $L^1(X,\mathcal{F},\mu)$}; \notag\\
& \text{for every } f\in L^1(X,\mathcal{B}_X,\mu),\
\mathbb{E}_{\mu}(f|\mathcal{F})(x)=\int_X f\,d\mu_{x}\ \text{for $\mu$-a.e. } x\in X. \label{meas3}
\end{align}
Then for any $f \in L^1(X,\mathcal{B}_X,\mu)$, one has
\begin{equation*}
\int_X \left(\int_X f(y)\,d\mu_x(y) \right)\, d\mu(x)=\int_X f \,d\mu.
\end{equation*}

The support of $\mu\in \mathcal{M}(X)$ is defined to be the set of
all points $x$ in $X$ for which every open neighborhood $U$ of $x$
has positive measure, that is
\begin{align*}
\mathrm{supp} (\mu)&= \{ x \in X :\ \mu(U) > 0\text{ for every open neighborhood $U$ of $x$}\}\\
&=X \setminus \bigcup\{U\subset X:\ U \text{ is open and } \mu(U)=0\}.
\end{align*}

To prepare the proof of the main results in the next section we state
or prove some useful lemmas in this section. The first one is the
following.

\begin{lem}\label{sigma-bij}
Let $X$ be a compact metric space, $\mu\in \mathcal{M}(X)$ and
$\mathcal{B}_\mu$ be the completion of $\mathcal{B}_X$ under $\mu$.
If $\mathcal{F}_1, \mathcal{F}_2$ are  two sub-$\sigma$-algebras of
$\mathcal{B}_\mu$ with $\mathcal{F}_1\supseteq \mathcal{F}_2$ and
\[\mu=\int_X \mu_{i,x}d\mu(x)\]
is the disintegration of $\mu$ over
$\mathcal{F}_i$ for $i=1,2$, then $\supp(\mu_{1,x})\subseteq
\supp(\mu_{2,x})$ for $\mu$-a.e. $x\in X$.
\end{lem}
\begin{proof}
First, choose $X_0\in \mathcal{B}_\mu$ with $\mu(X_0)=1$ such that
$\mu_{1,x},\mu_{2,x}\in \mathcal{M}(X)$ are well defined for all
$x\in X_0$. Since $X$ is a compact metric space, there exists a
countable base $\Gamma$ of the topology of $X$. For each $U\in
\Gamma$ and $i=1,2$, put
$$E_i(U)=\{x\in X_0: \mu_{i,x}(U)=0\}.$$
By \eqref{meas1} and \eqref{meas3}, we have $E_i(U)\in \mathcal{B}_\mu$ for $i=1,2$.

Next we are going to show $\mu(E_2(U)\setminus E_1(U))=0$.
By \eqref{meas3}, there exists
$\widetilde{E_2(U)}\in \mathcal{F}_2$ with
$\mu(\widetilde{E_2(U)}\Delta E_2(U))=0$. Moreover, combining this
with \eqref{meas1} and \eqref{meas3}, one has
\begin{align*}
\int_{E_2(U)}\mu_{1,x}(U)d\mu(x)&=\int_{\widetilde{E_2(U)}}\mu_{1,x}(U)d\mu(x)=
\int_{\widetilde{E_2(U)}}\mathbb{E}_\mu(1_U|\mathcal{F}_1)(x)d\mu(x)\\
&=\int_{\widetilde{E_2(U)}}\mathbb{E}_\mu\big(
\mathbb{E}_\mu(1_U|\mathcal{F}_1)|\mathcal{F}_2\big)(x) d\mu(x)\\
&=\int_{\widetilde{E_2(U)}}\mathbb{E}_\mu(1_U|\mathcal{F}_2)(x)d\mu(x)=\int_{\widetilde{E_2(U)}}\mu_{2,x}(U)
d\mu(x)\\
&=\int_{E_2(U)}\mu_{2,x}(U) d\mu(x)=0.
\end{align*}
This implies that $\mu(E_2(U)\setminus E_1(U))=0$.

Finally, we let $X_1=X_0\setminus \bigcup_{U\in
\Gamma}\big(E_2(U)\setminus E_1(U)\big)$. Then $\mu(X_1)=1$, and it
is clear that for any $x\in X_1$ and $U\in \Gamma$, $x\in E_2(U)$
implies $x\in E_1(U)$. Thus for any $x\in X_1$,
\begin{align*}
\supp(\mu_{2,x})=X\setminus \bigcup\limits_{U\in \Gamma,\ x\in
E_2(U)}U\supseteq X\setminus \bigcup\limits_{U\in \Gamma,\ x\in
E_1(U)}U=\supp(\mu_{1,x}).
\end{align*}
This completes the proof.
\end{proof}

To state the next lemma which is Lemma 4 in~\cite{BHS} we need
some notions. Let $(X,T)$ be a t.d.s., $\mu\in \mathcal{M}(X,T)$ and
$\mathcal{B}_\mu$ be the completion of $\mathcal{B}_X$ under $\mu$.
The {\it Pinsker $\sigma$-algebra} $P_\mu(T)$ is defined as the
smallest sub-$\sigma$-algebra of $\mathcal{B}_\mu$ containing $\{
\xi\in \mathcal{P}_X: h_\mu(T,\xi)=0\}$. It is well known that
$P_\mu(T)=P_\mu(T^{-1})$ and $P_\mu(T)$ is $T$-invariant, i.e.
$T^{-1}P_\mu(T)=P_\mu(T)$.

\begin{lem} \label{pinsker-con}
Let $(X,T)$ be a t.d.s. and $\mu\in \mathcal{M}^e(X,T)$. Then there
exists a sequence of partitions $\{W_i\}_{i=1}^{+\infty}$ in
$\mathcal{P}_X$ and $0=k_1<k_2<\cdots$ such that
\begin{enumerate}
\item $\lim\limits_{i\rightarrow+\infty} \text{diam}(W_i)=0$.
\item  $\lim\limits_{n\rightarrow+\infty} H_\mu(P_n|{\mathcal{P}}^-)=h_\mu(T)$,
where $P_n=\bigvee \limits_{i=1}^n T^{-k_i}W_i$ and $\mathcal{P}=\bigvee_{n=1}^\infty  P_n$.
\item $\bigcap \limits_{n=0}^{+\infty} \widehat{T^{-n}\mathcal{P}^-}=P_\mu(T)$.
\item $(T^{-n}\mathcal{P}^-)(x)\subseteq W^s(x,T)$ for each $n\in \mathbb{N}\cup \{0 \}$ and $x\in X$,
where $(T^{-n}\mathcal{P}^-)(x)$ is the atom of $T^{-n}\mathcal{P}^-$ containing $x$.
\end{enumerate}
\end{lem}
\begin{proof} The proof of the lemma follows directly from that of Lemma 4 in~\cite{BHS}.
For completeness, we outline the construction of $\{ W_i\}_{i=1}^{+\infty}\subset \mathcal{P}_X$
and $0=k_1<k_2<\cdots$. Let $\{ W_i\}_{i=1}^{+\infty}$ be an increasing sequence of finite
partitions of $X$ such that $\lim \limits_{i\rightarrow +\infty}
\text{diam}(W_i)=0$. Take $k_1=0$, then we find inductively
$k_1,k_2,\cdots$ such that for each $m\ge 2$, one has
$$H_{\mu}(P_n|P_{m-1}^-)-H_\mu(P_n|P_m^-)<\frac{1}{n} \frac{1}{2^{m-n}},
n=1,2,\cdots,m-1.$$
It is not hard to check that (1)--(4) hold (see for
example~\cite{Pa} or~\cite{Gl1}).
\end{proof}

We remark that the partition $\mathcal{P}$ constructed in
Lemma~\ref{pinsker-con} is an excellent partition (see \cite{BHS} or
\cite{Pa}), which is a key tool in the proof of our main results.

Let $(X,T)$ be a t.d.s. and $n\geq 2$. The $n$-th fold product
system of $(X,T)$ is denoted by $(X^{(n)}, T^{(n)})$, where
$X^{(n)}=X\times X\times \dotsb \times X$ ($n$-times) and
$T^{(n)}=T\times T\times \dotsb \times T$ ($n$-times). And we set
the diagonal of $X^{(n)}$ as  $\Delta_n=\{(x, x,\dotsc, x)\in
X^{(n)} :\ x\in X\}$, and set $\Delta^{(n)}=\{(x_1,
x_2,\dotsc,x_n)\in X^{(n)}:\ \text{there exist }1\leq i<j\leq
n\text{ with }x_i=x_j\}$.

Let $X$ be a metric space. A subset $K \subset X$ is called \emph{a
Mycielski set} if it is a union of countably many Cantor sets. This
definition was introduced in~\cite{BGKM}. Note that in~\cite{Ak} a
Mycielski set is required to be dense. For convenience we restate
here a version of Mycielski's theorem (see~\cite[Theorem~1]{My})
which we shall use. See~\cite{Ak} for a comprehensive treatment of
this topic.

\begin{lem}[Mycielski] \label{Myc}
Let $X$ be a perfect compact metric space. Assume that for every $n\geq 2$,
$R_n$ is a dense $G_\delta$ subset of $X^{(n)}$.
Then  there exists a dense Mycielski subset $K$ of $X$  such that
for every $n\geq 2$,
\[K^{(n)}\subset R_n\cup \Delta^{(n)}.\]
\end{lem}

Let $(X,T)$ be a t.d.s., $n\geq 2$ and $\eta>0$. Recall that a tuple
$(x_1,\dotsc,x_n)\in X^{(n)}$ is \emph{mean Li-Yorke $n$-scrambled
with modulus $\eta$} if
\[\liminf_{N\to +\infty}\frac{1}{N}\sum_{k=1}^N\max_{1\leq i<j\leq n}d\big(T^k x_i , T^k x_j)=0\]
and
\[\limsup_{N\to +\infty}\frac{1}{N}\sum_{k=1}^N\min_{1\leq i<j\leq n}d\big(T^k x_i , T^k x_j)\geq \eta.\]
Denote by $MLY_{n,\eta}(X,T)$ the set of all mean Li-Yorke
$n$-scrambled tuples with modulus $\eta$ in $(X,T)$, and
$MLY_n(X,T)=\bigcup_{\eta>0}MLY_{n,\eta}(X,T)$.

\begin{lem} \label{Mean-LY-Gdelata}
Let $(X,T)$ be a t.d.s., $n\geq 2$ and $\eta>0$. Then
$MLY_{n,\eta}(X,T)$ is a $G_\delta$ subset of $X^{(n)}$.
\end{lem}
\begin{proof}
Let
\[P_{n}(X,T)=\bigcap_{m=1}^{+\infty}\bigcap_{\ell=1}^{+\infty}
\Bigl(\bigcup_{N\ge\ell} \bigl\{(x_1,\dotsc,x_n)\in X^{(n)}:
\frac{1}{N}\sum_{k=0}^{N-1}\max_{1\leq i<j\leq n}d\big(T^k x_i , T^k x_j) <\tfrac{1}{m}\bigr\}\Bigr)\]
and
\[D_{n,\eta}(X,T)=\bigcap_{m=1}^{+\infty}\bigcap_{\ell=1}^{+\infty}
\Bigl(\bigcup_{N\ge \ell} \bigl\{(x_1,\dotsc,x_n)\in X^{(n)}:
\frac{1}{N}\sum_{k=0}^{N-1}\min_{1\leq i<j\leq n}d\big(T^k x_i , T^k
x_j)>\eta-\tfrac{1}{m}\bigr\}\Bigr).\] It is easy to check that
$P_n(X,T)$ and $D_{n,\eta}(X,T)$ are $G_\delta$ subsets of
$X^{(n)}$. Then so is $MLY_{n,\eta}(X,T)$, since
$MLY_{n,\eta}(X,T)=P_n(X,T)\bigcap D_{n,\eta}(X,T)$.
\end{proof}

\section{Proofs of the main results}
In this section first we study the stable sets and unstable sets in
positive entropy systems, then give the proofs of the main results. Finally
we show that our main results hold for continuous maps with an obvious
modification.

\subsection{Stable sets and unstable sets in positive entropy systems}
First, we have the following lemma.
\begin{lem} \label{lem-1}
Let $(X,T)$ be a t.d.s. and $\mu\in \mathcal{M}^e(X,T)$ with
$h_\mu(T)>0$. If
\[\mu=\int_X \mu_x d \mu(x)\]
is the disintegration of $\mu$ over the Pinsker $\sigma$-algebra $P_\mu(T)$,
then for $\mu$-a.e. $x\in X$,
\begin{align*}
\overline{W^s(x,T)\cap \supp(\mu_x)}=\supp(\mu_x) \text{ and }
\overline{W^u(x,T)\cap\supp(\mu_x)}=\supp(\mu_x).
\end{align*}
\end{lem}

\begin{proof}
Since $P_\mu(T)$ is also the Pinsker $\sigma$-algebra of the system
$(X,\mathcal{B}_\mu,\mu,T^{-1})$ and $W^s(x,T^{-1})=W^u(x,T)$, by
symmetry it suffices to show that for $\mu$-a.e. $x\in X$,
\[\overline{W^s(x,T)\cap \supp(\mu_x)}=\supp(\mu_x). \]
By Lemma~\ref{pinsker-con}, there exist $\{ W_i\}_{i=1}^{+\infty}\subset \mathcal{P}_X$
and $0=k_1<k_2<\cdots$ satisfying that
\begin{enumerate}
\item $\lim\limits_{i\rightarrow+\infty} \text{diam}(W_i)=0$.
\item  $\lim\limits_{n\rightarrow+\infty} H_\mu(P_n|{\mathcal{P}}^-)=h_\mu(T)$,
where $P_n=\bigvee \limits_{i=1}^n T^{-k_i}W_i$ and $\mathcal{P}=\bigvee_{n=1}^\infty  P_n$.
\item $\bigcap \limits_{n=0}^{+\infty} \widehat{T^{-n}\mathcal{P}^-}=P_\mu(T)$.
\item $(T^{-n}\mathcal{P}^-)(x)\subseteq W^s(x,T)$ for each $n\in \mathbb{N}\cup \{0 \}$ and $x\in
X$.
\end{enumerate}

For every $n\ge 0$, let \[\mu=\int_X \mu_{n,x} d \mu(x)\]
be the disintegration of $\mu$ over $\widehat{T^{-n}\mathcal{P}^-}$.
Then for every $n\geq 0$,
\begin{equation}\label{eq-eee}
\mu_{n,x}((T^{-n}\mathcal{P}^-)(x))=1 \text{ and so }\mu_{n,x}(W^s(x,T))=1\
\text{ for $\mu$-a.e. $x\in X$}.
\end{equation}
Moreover, since $$\widehat{\mathcal{P}^-}\supset
\widehat{T^{-1}\mathcal{P}^-}\supset
\widehat{T^{-2}\mathcal{P}^-}\supset \cdots\  \text{and}\
\bigcap_{n=0}^{+\infty} \widehat{T^{-n}\mathcal{P}^-}=P_\mu(T),$$
there exists a set $X_1\in\mathcal{B}_\mu$ with $\mu(X_1)=1$ such
that for any $x\in X_1$, $\lim \limits_{n\rightarrow +\infty}
\mu_{n,x}=\mu_x$ under the weak$^*$ topology (see for example
\cite[Corollary 5.21]{EW}), and for each $n\geq 0$,
\begin{align}\label{eq-eee-1}
\mu_{n,x}(W^s(x,T))=1 \text{ and }\supp(\mu_{n,x})\subseteq
\supp(\mu_{n+1,x})\subseteq \supp(\mu_x) \end{align} by Lemma
\ref{sigma-bij} and \eqref{eq-eee}.

We will show that for every $x\in X_1$, $\supp(\mu_x)=\overline{W^s(x,T)\cap \supp(\mu_{x})}$.

Fix $x\in X_1$. It is clear that $\overline{W^s(x,T)\cap
\supp(\mu_{x})}\subseteq \supp(\mu_x)$. By \eqref{eq-eee-1}, one has
that for any $n\ge 0$
$$\mu_{n,x}(\overline{W^s(x,T)\cap \supp(\mu_{x})})\ge \mu_{n,x}(W^s(x,T)\cap
\supp(\mu_{n,x}))=1.$$ Since $\lim \limits_{n\rightarrow +\infty}
\mu_{n,x}=\mu_x$ under the weak$^*$ topology,
$$\mu_{x}(\overline{W^s(X,T)\cap\supp(\mu_{x})})\geq
\limsup \limits_{n\rightarrow \infty}\mu_{n,x}(\overline{W^s(X,T)\cap \supp(\mu_{x})})=1$$
and then
$\supp(\mu_x)\subset \overline{W^s(x,T)\cap\supp(\mu_{x})}$.
\end{proof}

Now we give the definition of entropy tuples introduced in
\cite{HY}, see \cite{Bl} and \cite{BGH} for the pair case
respectively.

\begin{de}
Let $(X,T)$ be a t.d.s. and $n\geq 2$.
\begin{enumerate}
\item A cover $\mathcal{U}=\{U_1,\dotsc,U_k\}$ of $X$
is said to be \emph{admissible} with respect to $(x_1,\dotsc,x_n)\in X^{(n)}$ if
for each  $1\leq i\leq k$ there exists $j_i$ such that $x_{j_i}\not\in \overline{U_i}$.
\item A tuple $(x_1,\dotsc,x_n)\in X^{(n)}$ is called a \emph{topological entropy $n$-tuple},
if at least two of the points in $(x_1,\dotsc,x_n)$ are distinct and
any admissible open cover with respect to  $(x_1,\dotsc,x_n)$ has positive topological entropy.
Denote by $E_n(X,T)$ the set of all topological entropy $n$-tuples.
\item A tuple $(x_1,\dotsc,x_n)\in X^{(n)}$ is called an
\emph{entropy $n$-tuple for $\mu\in \mathcal{M}(X,T)$},
if  at least two of the points in $(x_1,\dotsc,x_n)$ are distinct and
any admissible Borel partition $\alpha$ with respect to $(x_1,\dotsc,x_n)$
has positive measure-theoretic entropy.
Denote by $E_n^\mu(X,T)$ the set of all entropy $n$-tuples for $\mu$.
\end{enumerate}
\end{de}

Let $(X,T)$ be a t.d.s. and $\mu\in \mathcal{M}^{}(X,T)$. For every
$n\geq 2$, define a measure $\lambda_n(\mu)$ on $(X^{(n)},T^{(n)})$
by letting
$$\lambda_n(\mu)=\int_X \mu_x^{(n)} d \mu (x),$$
where $\mu_x^{(n)}=\mu_x\times \mu_x\times \dotsb \times \mu_x$
($n$-times). It is well known that when $\mu$ is ergodic with
positive entropy (see for example \cite{BGKM,HY}), ${\mu}_x$ is
non-atomic for $\mu$-a.e $x\in X$ and $\lambda_n(\mu)$ is a
$T^{(n)}$-invariant ergodic measure on $X^{(n)}$. The relation
between topological entropy tuples and entropy tuples for $\mu$ is
the following.
\begin{pro}[\cite{BGH, Gl,HY}]\label{thm:HY-06}
Let $(X,T)$ be a t.d.s. and $\mu\in \mathcal{M}^e(X,T)$. Then  for
every $n\geq 2$,
\begin{enumerate}
  \item $E_n^\mu(X,T) =\supp(\lambda_n(\mu))\setminus\Delta_n$, and
  \item $E_n(X,T)=\overline{\bigcup_{\mu\in {\mathcal M}^e(X,T)} E^{\mu}_n(X,T)}\setminus \Delta_n.$
\end{enumerate}
\end{pro}

Let $(X,T)$ be a t.d.s. and $n\geq 2$. A tuple $(x_1,\dotsc,x_n)\in
X^{(n)}$ is called \emph{$n$-asymptotic} if
\[\lim_{k\to+\infty}\max_{1\leq i<j\leq n} d(T^kx_i,T^kx_j)=0.\]
The set of all asymptotic $n$-tuples is denoted by $Asy_n(X,T)$. It
was shown in \cite{HY02} that $Asy_2(X,T)$ is of first category if
$(X,T)$ is sensitive. It is known \cite{Gl, HY} that if $\mu$ is
ergodic with positive entropy then $(E_2^\mu(X,T)\cup
\Delta_2,T\times T)$ is transitive and
$\lambda_2(\mu)((E_2^\mu(X,T)\cup \Delta_2)\cap Asy_2(X,T) )=0$
since an asymptotic pair cannot be a transitive point of
$(E_2^\mu(X,T)\cup \Delta_2,T\times T)$. Note that Theorem
\ref{thm-2}(1) states that for every $n\geq 2$, $Asy_n(X,T)\cap
E_n(X,T)$ is dense in $E_n(X,T)$.


\begin{proof}[Proof of Theorem \ref{thm-2}(1)]
Let $\mu$ be an ergodic invariant measure with positive entropy.
Then by Lemma~\ref{lem-1}, there exists $X_1\in{\mathcal{B}_X}$ with
$\mu(X_1)=1$ such that for each $x\in X_1$, ${W^s(x,T)\cap
\supp(\mu_x)}$ is dense in $\supp(\mu_x)$. It is clear that
$W^s(x,T)^{(n)}\subset Asy_n(X,T)$ and thus
$Asy_n(X,T)\cap\supp(\mu_x)^{(n)}\supset W^s(x,T)^{(n)}\cap
\supp(\mu_x)^{(n)}\supset (W^s(x,T)\cap \supp(\mu_x))^{(n)}$ which
implies that
\begin{equation}\label{first}\overline{Asy_n(X,T)\cap\supp(\mu_x)^{(n)}}=\supp(\mu_x)^{(n)}
\ \text{for each}\ x\in X_1.
\end{equation}

By Proposition~\ref{thm:HY-06}(1), $E_n^\mu(X,T)
=\supp(\lambda_n(\mu))\setminus\Delta_n$, so we have
$\lambda_n(\mu)(E_n^\mu(X,T)\cup\Delta_n)=1$. This implies that
there exists $X_2\in {\mathcal{B}_X} $ with $\mu(X_2)=1$ such that
for each $x\in X_2$, one has
$\mu_x^{(n)}(E_n^\mu(X,T)\cup\Delta_n)=1$ by the definition of
$\lambda_n(\mu)$. Thus, $\supp(\mu_x)^{(n)}=\supp\big( \mu_x^{(n)}\big)\subset
E_n^\mu(X,T)\cup\Delta_n$ and hence $\supp(\mu_x)^{(n)}\setminus
\Delta_n\subset E_n^\mu(X,T)$. Now by (\ref{first}) and the fact
that $\mu_x$ is non-atomic, we have for each $x\in X_1\cap X_2$
$$\overline{Asy_n(X,T)\cap E^\mu_n(X,T)}\supset \overline{Asy_n(X,T)\cap
(\supp(\mu_x)^{(n)}\setminus \Delta_n)}=\supp(\mu_x)^{(n)}$$
which implies that
$$\overline{Asy_n(X,T)\cap E^\mu_n(X,T)}\supset \bigcup_{x\in X_1\cap X_2}\supp(\mu_x)^{(n)}.$$ Thus we get that
$$\lambda_n({\mu})\big(\overline{Asy_n(X,T)\cap E^\mu_n(X,T)}\big)\ge
\int_X \mu_x^{(n)}\Big(\bigcup_{x\in X_1\cap X_2}\supp(\mu_x)^{(n)}\Big) d\mu(x)=1$$ which
implies that
\begin{equation}\label{mid}\overline{Asy_n(X,T)\cap E^\mu_n(X,T)}\supset
\supp(\lambda_n(\mu))\supset E_n^\mu(X,T),\end{equation} i.e.
$Asy_n(X,T)\cap E^\mu_n(X,T)$ is dense in $E^\mu_n(X,T)$.

By Proposition~\ref{thm:HY-06}(2) and (\ref{mid}), for each $\mu\in
{\mathcal M}^e(X,T)$ we have
$$\overline{Asy_n(X,T)\cap E_n(X,T)}\supset \overline{Asy_n(X,T)\cap
E_n^\mu(X,T)}=E_n^\mu(X,T),$$
which implies that
$$\overline{Asy_n(X,T)\cap
E_n(X,T)}\supset \overline{\bigcup_{\mu\in {\mathcal
M}^e(X,T)}E_n^\mu(X,T)}\supset E_n(X,T),$$ i.e. $Asy_n(X,T)\cap
E_n(X,T)$ is dense in $E_n(X,T)$.
\end{proof}

\begin{rem} We have defined an asymptotic pair in the positive
direction. It is clear that we may define a two-sided asymptotic
pair $(x,y)$ by $\lim_{n\rightarrow\pm\infty}d(T^nx,T^ny)=0$. It is
known that positive entropy does not imply the existence of a
non-diagonal two-sided asymptotic pair, see \cite[Example 3.4]{LS}
by Lind and Schmidt. In \cite{CL} Chung and Li showed that if
$\Gamma$ is a polycyclic-by-finite group and $\Gamma$ acts on a
compact Abelian group $X$ expansively by automorphisms, then $\{x\in
X: \lim_{\Gamma\ni s\to\infty}(sx,se)=0\}$
is dense in $\{x\in X:(x,e)\in E_2(X,\Gamma)\}$, where $e$ is the unit of $\Gamma$.
\end{rem}

\subsection{Proofs of Theorems \ref{thm-1} and \ref{thm-2}(2)}
\begin{proof}[Proof of Theorem \ref{thm-1}]   Let $\mathcal{B}_{\mu}$ be the  completion of
$\mathcal{B}_X$ under $\mu$. Then $(X,\mathcal{B}_\mu, \mu,T)$ is a
Lebesgue system.  Let $P_\mu(T)$ be the Pinsker $\sigma$-algebra of
$(X,\mathcal{B}_\mu,\mu,T)$. Let \[\mu=\int_X \mu_x d \mu(x)\] be
the disintegration of $\mu$ over $P_\mu(T)$. By Lemma~\ref{lem-1},
there exists a set $X_1\in\mathcal{B}_\mu$ with $\mu(X_1)=1$ such
that for any $x\in X_1$,
\begin{align}\label{assup}
\overline{W^s(x,T)\cap \supp(\mu_x)}=\supp(\mu_x) \text{
and }\overline{W^s(x,T^{-1})\cap
\supp(\mu_x)}=\supp(\mu_x).
\end{align}

For every $n\geq 2$, let $\lambda_n(\mu)$ be the measure on
$(X^{(n)},T^{(n)})$ defined before. 
Recall that $\Delta^{(n)}=\{(x_1, x_2,\dotsc,x_n)\in X^{(n)}:\
\text{there exist }1\leq i<j\leq n\text{ with }x_i=x_j\}$. Since
$\mu_x$ is non-atomic for a.e. $x\in X$, by the Fubini theorem,
$\lambda_n(\mu)(\Delta^{(n)})=0$. It is easy to check that
$$X^{(n)}\setminus \Delta^{(n)}=\bigcup_{k=1}^\infty \{(x_1,\dotsc,x_n)\in X^{(n)}: \min_{1\leq i<j\leq n}
d(x_i,x_j)>\tfrac{1}{k}\}.$$ Thus there exists $\tau>0$ such that
$\lambda_n(\mu)(W_n)>0$, where
$$W_n=\{(x_1,\dotsc,x_n)\in X^{(n)}: \min_{1\leq i<j\leq n}
d(x_i,x_j)>\tau\}.$$ Let $\eta_n=\tau\lambda_n(\mu)(W_n)$, and let
$G_n^+$ be the set of all generic points of $\lambda_n(\mu)$ for
$T^{(n)}$, that is $(x_1,\dotsc,x_n)\in G_n^+$ if and only if
$$\frac{1}{N}\sum_{i=0}^{N-1}\delta_{(T^{(n)})^i(x_1,\dotsc,x_n)}\rightarrow \lambda_n(\mu) $$
under the weak$^*$-topology, where $\delta_y$ is the point mass on $y$.
Then $\lambda_n(\mu)(G_n^+)=1$.
For every $(x_1,\dotsc,x_n)\in G_n^+$,
\begin{align*}
\limsup_{N\rightarrow \infty} \frac{1}{N}\sum_{k=0}^{N-1}\min_{1\leq i<j\leq n}d(T^kx_i,T^kx_j)
&\geq \limsup_{N\rightarrow \infty} \frac{1}{N}\sum_{k=0}^{N-1} \tau\delta_{(T^{(n)})^k(x_1,\dotsc,x_n)}(W_n)\\
&\geq \tau\liminf_{N\rightarrow \infty} \frac{1}{N}\sum_{k=0}^{N-1}\delta_{(T^{(n)})^k(x_1,\dotsc,x_n)}(W_n)\\
&\geq \tau \lambda_n(\mu)(W_n)=\eta_n.
\end{align*}
This shows that $G_n^+\subset D_{n,\eta_n}(X,T)$, where
$D_{n,\eta_n}$ is defined in the proof of
Lemma~\ref{Mean-LY-Gdelata}. Similarly, if we let $G_n^-$ be the set
of all generic points of $\lambda_n(\mu)$ for $(T^{-1})^{(n)}$,
then $\lambda_n(\mu)(G_n^-)=1$ and $G_n^-\subset
D_{n,\eta_n}(X,T^{-1})$.

Since $\mu_x$ is non-atomic for $\mu$-a.e. $x\in X$ and
\[1=\lambda_n(\mu)(G_n^+\cap G_n^-)=\int_X \mu_x^{(n)}(G_n^+\cap G_n^-)d \mu (x),\]
there exists a set $X_2\in \mathcal{B}_\mu$ with $\mu(X_2)=1$ such
that $\mu_x$ is non-atomic and $\mu_x^{(n)}(G_n^+\cap G_n^-)=1$ for
all $x\in X_2$  and all $n\geq 2$. Now let $X_0=X_1\cap X_2$. Then
$\mu(X_0)=1$.

Now fix $x\in X_0$. Then $\supp(\mu_x)$ is a perfect closed subset
of $X$, since $\mu_x$ is non-atomic. By the construction of $X_2$,
we have that $\mu_x^{(n)}(G_n^+\cap G_n^-\cap\supp(\mu_x)^{(n)})=1$, and
then $G_n^+\cap G_n^-\cap\supp(\mu_x)^{(n)}$ is dense in
$\supp(\mu_x)^{(n)}$. So $$D_{n,\eta_n}(X,T)\cap
D_{n,\eta_n}(X,T^{-1})\cap\supp(\mu_x)^{(n)}$$ is a dense $G_\delta$
subset of $\supp(\mu_x)^{(n)}$ by Lemma~\ref{Mean-LY-Gdelata}, and
the facts that $G_n^+\subset D_{n,\eta_n}(X,T)$ and $G_n^-\subset
D_{n,\eta_n}(X,T^{-1})$. By the construction of $X_1$, we have
$Asy_n(X,T)\cap \supp(\mu_x)^{(n)}$ and $Asy_n(X,T^{-1})\cap
\supp(\mu_x)^{(n)}$ are dense in  $\supp(\mu_x)^{(n)}$.
Recall that $P_n(X,T)$ is defined in the proof of Lemma~\ref{Mean-LY-Gdelata}.
Then
$$P_n(X,T)\cap P_n(X,T^{-1})\cap \supp(\mu_x)^{(n)}$$ is a dense
$G_\delta$ subset of $\supp(\mu_x)^{(n)}$
by Lemma~\ref{Mean-LY-Gdelata} and the facts that $Asy_n(X,T)\subset
P_n(X,T)$ and $Asy_n(X,T^{-1})\subset P_n(X,T^{-1})$.
Applying Lemma~\ref{Mean-LY-Gdelata} again, we have that $$MLY_{n,\eta_n}(X,T)\cap
MLY_{n,\eta_n}(X,T^{-1})\cap\supp(\mu_x)^{(n)}$$ is a dense
$G_\delta$ subset of  $\supp(\mu_x)^{(n)}$. By Mycielski's theorem
(Lemma \ref{Myc}), there exists a dense Mycielski subset $K_x$  of
$\supp(\mu_x)$ such that for every $n\geq 2$,
\[K_x^{(n)}\subset \big(MLY_{n,\eta_n}(X,T)\cap MLY_{n,\eta_n}(X,T^{-1})\big)\cup \Delta^{(n)}.\]
Then $K_x$ is as required.
\end{proof}

Recall that Theorem \ref{thm-2}(2) states that for every $n\geq 2$,
$MLY_n(X,T)\cap E_n(X,T)$ is dense in $E_n(X,T)$.
\begin{proof}[Proof of Theorem \ref{thm-2}(2)]
Let $\mu$ be an ergodic invariant measure with positive entropy. By
the proof of Theorem~\ref{thm-1} there exists $\eta_n>0$ such that
for $\mu$-a.e. $x\in X$, $MLY_{n,\eta_n}(x,T)\cap\supp(\mu_x)^{(n)}$
is dense in $\supp(\mu_x)^{(n)}$. By the same argument as in the
proof of Theorem \ref{thm-2}(1), we get the result.
\end{proof}

\begin{cor}\label{cor:K-MLY}
Let $(X,T)$ be a t.d.s. If there is an invariant measure $\mu$ of
full support such that $(X,\mathcal B_X,\mu,T)$ is a Kolmogorov
system, then there exists a dense Mycielski subset $K$ of $X$ such
that for any $n\geq 2$, $K$ is mean Li-Yorke $n$-scrambled (with
modulus $\eta_n>0$).
\end{cor}
\begin{proof}
If $(X,\mathcal B_X,\mu,T)$ is a Kolmogorov system, then the Pinsker $\sigma$-algebra
$P_\mu(T)=\{\emptyset,X\}$ (mod $\mu$).
So the disintegration of $\mu$ over $P_\mu(T)$ is trivial, that is for $\mu$-a.e. $x\in X$,
$\mu_x=\mu$. Now the result follows from Theorem~\ref{thm-1}, since $\supp(\mu_x)=\supp(\mu)=X$.
\end{proof}

Recall that we say a t.d.s. $(X,T)$ has \emph{uniformly positive
entropy} (see \cite{Bl}) if any cover of $X$ by two non-dense open
sets has positive entropy. It is not hard to see that a t.d.s.
$(X,T)$ has uniformly positive entropy if and only if $E_2(X,T)\cup
\Delta_2=X^{(2)}$.

\begin{cor}
If a uniformly positive entropy system $(X,T)$ admits an ergodic
invariant measure $\mu$ with full support and $h_\mu(T)>0$, then
there exists a dense Mycielski subset $K$ of $X$ such that $K$ is
mean Li-Yorke $2$-scrambled (with modulus $\eta>0$).
\end{cor}
\begin{proof}
By Theorem\ref{thm-2}(2), there exists $\eta_0>0$ such that
$MLY_{2,\eta_0}(X,T)$ is dense in $E^\mu_2(X,T)$. Since $\mu$ has
full support, $\Delta_2\subset \overline{E^\mu_2(X,T)}\subset
\overline{MLY_{2,\eta_0}(X,T)}$. Let $\eta=\eta_0/2$.

First, we show that $D_{2,\eta}(X,T)$ is dense in $X^{(2)}$.
For any $(x_1,x_2)\in X^{(2)}$ and $\ep>0$, there exists $(x_1',x_1'')\in MLY_{2,\eta_0}(X,T)$ with
$d(x_1,x_1')<\ep$ and $d(x_1,x_1'')<\ep$. Then
\begin{align*}
\eta_0&\leq \limsup_{N\rightarrow \infty} \frac{1}{N}\sum_{k=0}^{N-1}d(T^kx_1',T^kx_1'')\\
&\leq \limsup_{N\rightarrow \infty} \frac{1}{N}\sum_{k=0}^{N-1} \big(d(T^kx_1',T^kx_2)+d(T^kx_1'',T^kx_2)\big)\\
&\leq \limsup_{N\rightarrow \infty} \frac{1}{N}\sum_{k=0}^{N-1} d(T^kx_1',T^kx_2)+
\limsup_{N\rightarrow \infty} \frac{1}{N}\sum_{k=0}^{N-1} d(T^kx_1'',T^kx_2)
\end{align*}
Then either $(x_1',x_2)\in D_{2,\eta}(X,T)$ or $(x_1'',x_2)\in D_{2,\eta}(X,T)$.
This shows that $D_{2,\eta}(X,T)$ is dense in $X^{(2)}$.

Since $(X,T)$ has uniformly positive entropy, by
Theorem\ref{thm-2}(1), $Asy_2(X,T)$ is dense in $X^{(2)}$, and then
$P_2(X,T)$ is also dense in $X^{(2)}$. Therefore,
$MLY_{2,\eta}(X,T)=P_2(X,T)\cap D_{2,\eta}(X,T)$ is a dense
$G_\delta$ subset of $X^{(2)}$. By Mycielski theorem, there exists a
dense Mycielski subset $K$  of $X$ such that $K^2\subset
MLY_{2,\eta}(X,T)\cup \Delta_2$. Then $K$ is a mean Li-Yorke
$2$-scrambled set with modulus $\eta$.
\end{proof}

\subsection{Non-invertible case}
In this subsection, we will generalize the main results to the
non-invertible case. Let $(X,T)$ be a non-invertible t.d.s., i.e.
$X$ is a compact metric space, and $T:X\rightarrow X$ is a
continuous surjective map but not one-to-one.

For a t.d.s. $(X,T)$ with metric $d$, we say that
$(\widetilde{X},\widetilde{T})$ is the \emph{natural extension} of
$(X,T)$, if $\widetilde{X}=\{ (x_1,x_2,\cdots): T(x_{i+1})=x_i,\
i\in \mathbb{N} \}$ is a subspace of the product space
$X^{\mathbb{N}}=\prod\limits_{i=1}^\infty X$ endowed with the
compatible metric $\tilde d$ defined by
$$\tilde d((x_1,x_2,\cdots),(y_1,y_2,\cdots))=\sum_{i=1}^\infty
\frac{d(x_i,y_i)}{2^i},$$ and
$\widetilde{T}:\widetilde{X}\rightarrow \widetilde{X}$ is the shift
homeomorphism, i.e.
$\widetilde{T}(x_1,x_2,\cdots)=(T(x_1),x_1,x_2,\cdots)$. Let
$\pi:\widetilde{X}\rightarrow X$ be the projection to the first
coordinate. Then $\pi:(\widetilde{X},\widetilde{T})\rightarrow
(X,T)$ is a factor map. It is clear that a tuple $(\tilde
x_1,\dotsc,\tilde x_n)$ is asymptotic in $\tilde X$ if and only if
the tuple $(\pi(\tilde x_1),\dotsc,\pi(\tilde x_n))$ is asymptotic
in $X$, that is for every $n\geq 2$, $\pi^{(n)}(Asy_n(\tilde
X,\tilde T))=Asy_n(X,T)$.  Using the fact that for $\tilde x,\tilde
y\in \widetilde X$ and $k, M\geq 0$,
\begin{align*}
\frac 1 2d(T^k\pi(\tilde x),T^k\pi(\tilde y))&\leq
\tilde d(\widetilde T^k \tilde x, T^k\tilde y) \\
&\leq
\sum_{p=1}^{M+1} \frac{d(T^{k-p+1}\pi(\tilde x), T^{k-p+1}\pi(\tilde y))}{2^p}+\frac{\operatorname{diam}(X)}{2^{M}}
\end{align*}
the following lemma is easy to verify.

\begin{lem}\label{lem:pi-dist}
For every $\eta>0$, there exists a $\delta>0$ such that
if $(\tilde x_1,\dotsc,\tilde x_n)\in \widetilde{X}^{(n)}$ is mean Li-Yorke $n$-scrambled with modulus $\eta$
then $(\pi(\tilde x_1),\dotsc,\pi(\tilde x_n))\in {X}^{(n)}$ is mean Li-Yorke $n$-scrambled with modulus $\delta$.
\end{lem}

Thus we have
\begin{thm}\label{conti}
Let $(X,T)$ be a non-invertible t.d.s.  and $\mu\in
\mathcal{M}^e(X,T)$ with $h_\mu(T)>0$. Then
there exists a sequence of positive numbers $\{\eta_n\}_{n=2}^{+\infty}$
satisfying that for $\mu$-a.e. $x\in X$, there
exists a Mycielski subset $K_x\subseteq\overline{W^s(x,T)}$ such
that for any $n\geq 2$, $K_x$ is mean Li-Yorke $n$-scrambled with
modulus $\eta_n$.
\end{thm}
\begin{proof}
Let $(\widetilde{X},\widetilde{T})$ be the natural extension of
$(X,T)$ and $\pi:\widetilde{X}\rightarrow X$ be the projection to
the first coordinate. It is well known that there exists
$\widetilde{\mu}\in \mathcal{M}^e(\widetilde{X},\widetilde{T})$ such
that $\pi(\widetilde{\mu})=\mu$. Clearly,
$h_{\widetilde{\mu}}(\widetilde{T})\ge h_\mu(T)>0$. By
Theorem~\ref{thm-1}, there exists a Borel subset
$\widetilde{X}_0\subseteq \widetilde{X}$ with
$\widetilde{\mu}(\widetilde{X}_0)=1$ such that for any
$\widetilde{x}\in \widetilde{T}_0$, there exists a Mycielski subset
$K_{\widetilde{x}}$ of $\overline{W^s(\widetilde{x},\widetilde T)}$
satisfying the conclusion of Theorem~\ref{thm-1}.

Let $X_0=\pi(\widetilde{X}_0)$. Then $X_0$ is a $\mu$-measurable set and $\mu(X_0)=1$.
For any $x\in X_0$, there exists $\widetilde{x}\in \widetilde{X}_0$ such that
$\pi(\widetilde{x})=x$.  Let $K_x=\pi(K_{\widetilde{x}})$.
Then by Lemma~\ref{lem:pi-dist}, $K_x$ satisfies our requirements.
\end{proof}

\begin{thm}\label{cor-310} Let (X,T) be a non-invertible t.d.s. with positive entropy.
Then for every $n\geq 2$, $Asy_n(X,T)\cap E_n(X,T)$ and $MLY_{n}\cap
E_n(X,T)$ are dense in $E_n(X,T)$.
\end{thm}
\begin{proof}
Let $(\widetilde{X},\widetilde{T})$ be the natural extension of
$(X,T)$ and $\pi:\widetilde{X}\rightarrow X$ be the projection to
the first coordinate. Then  $Asy_n(X,T)=\pi^{(n)}(Asy_n(\widetilde
X,\widetilde T))$, $MLY_n(X,T)=\pi^{(n)}(MLY_n(\widetilde
X,\widetilde T))$ and $E_n(X,T)\subset \pi^{(n)}(E_n(\widetilde
X,\widetilde T))$ (see~\cite{HY}). Thus, the result follows from
Theorem~ \ref{thm-2}.
\end{proof}

\section{Final remarks}

Downarowicz and Lacroix~\cite{DY} showed that positive topological
entropy implies DC1$\frac 12$, which is delicately weaker than DC1.
We restate their result as follows.

\begin{thm}[\cite{DY}]
If a t.d.s. $(X,T)$ has positive entropy, then there exists an
uncountable subset $S$ of $X$ such that for any two distinct points
$x,y\in S$, one has
\begin{enumerate}
\item for every $t> 0$, the set
$\{k\in\mathbb{Z}_+:\ d(T^kx,T^ky)<t \}$ has upper density $1$, and
\item for every $s\in (0,1)$, there exists $t_{s}>0$ such that
the upper density of the set $\{k\in\mathbb{Z}_+:\ d(T^kx,T^ky)>t_{s}\}$ is at least $s$.
\end{enumerate}
\end{thm}

After a little modification of the proof of Theorem~\ref{thm-1}, we
can strengthen their result as follows.
\begin{thm}\label{thm-42}
Let $(X,T)$ be a t.d.s. If there exists an invariant ergodic measure
$\mu$ with $h_\mu(T)>0$, then for $\mu$-a.e. $x\in X$, there exists
a Mycielski set $K_x\subseteq \overline{W^s(x,T)}\cap
\overline{W^u(x,T)}$ such that for any $n\geq 2$ and any distinct
$n$ points $x_1,\dotsc,x_n\in K_x$, one has
\begin{enumerate}
\item for every $t>0$,  the set
$\{k\in\mathbb{Z}_+:\ \max_{1\leq i<j\leq n}d(T^kx,T^ky)<t \}$ has upper density $1$,  and
\item for every $s\in (0,1)$, there exists $t_{n,s}>0$ such that
the upper density of the set $\{k\in\mathbb{Z}_+:\ \min_{1\leq i<j\leq n}d(T^kx,T^ky)>t_{n,s}\}$
is at least $s$.
\end{enumerate}
\end{thm}
\begin{proof}[Sketch of the proof]
We keep the notation as in the  proof of Theorem~\ref{thm-1}. Since
$\mu_x$ is atomless for a.e. $x\in X$, by the Fubini theorem,
$\lambda_n(\mu)(\Delta^{(n)})=0$ and then
\[\lim_{\varepsilon\to 0} \lambda_n(\mu)([\Delta^{(n)}]_\varepsilon)=0,\]
where $[\Delta^{(n)}]_\varepsilon =\{(x_1,\dotsc,x_n)\in X^{(n)}:\
\max_{1\leq i<j\leq n}d(x_i,x_j)\leq\varepsilon\}$. For  every
$s\in (0,1)$, there exists $t_{n,s}>0$ such that
$\lambda_n(\mu)([\Delta^{(n)}]_{t_{n,s}})<1-s$. Then
\[\lambda_n(\mu)(X^{(n)}\setminus [\Delta^{(n)}]_{t_{n,s}})> s.\]

For every $n\geq 2$, let $A_n$ and $B_n$ be the set of all points
$(x_1,\dotsc,x_n)\in X^{(n)}$ satisfying the condition (1) and (2),
respectively. Clearly, for every $n\geq 2$, $Asy_n(X,T)\subset A_n$
and $G_n^+\subset B_n$. Then it is not hard to see that for every
$n\geq 2$ and $x\in X_0$, $A_n\cap B_n\cap\supp(\mu_x)^{(n)}$
contains a dense $G_\delta$ subset of  $\supp(\mu_x)^{(n)}$. By
Mycielski's theorem, there exists a dense Mycielski subset $K_x$  of
$\supp(\mu_x)$ such that for every $n\geq 2$, $K_x^{(n)}\subset
(A_n\cap B_n)\cup \Delta^{(n)}$. Then $K_x$ is as required.
\end{proof}

Recall that in \cite{KL} Kerr and Li gave a combinatorial proof of
the fact that positive entropy implies Li-Yorke's chaos. Since the
proofs in \cite{D} and in our paper are measure-theoretical, so
finding a topological or combinatorial proof of the fact that
positive entropy implies DC2 chaos or an uncountable mean Li-Yorke
$n$-scrambled set for $n\ge 2$ is natural. Unfortunately, we do not
have one at this moment.

\end{document}